\documentclass[11pt,a4paper]{article}

\title{\bf Quasi-invariance of the stochastic flow associated to It\^o's SDE
with singular time-dependent drift}
\author{Dejun Luo\footnote{Email: luodj@amss.ac.cn. Partly supported by the Key Laboratory of
RCSDS, CAS (2008DP173182), NSFC (11101407) and AMSS (Y129161ZZ1)}
\vspace{3mm}\\
{\footnotesize Institute of Applied Mathematics, Academy of Mathematics and Systems Science,}\\
{\footnotesize Chinese Academy of Sciences, Beijing 100190, China}
}
\date{}

\usepackage{amssymb,amsmath,amsfonts,amsthm,color}

\def\R{\mathbb{R}}
\def\E{\mathbb{E}}

\def\<{\langle}
\def\>{\rangle}
\newcommand{\eps}{\varepsilon}

\def\F{\mathcal{F}}
\def\H{\mathbb{H}}
\def\L{\mathbf{L}}
\def\P{\mathbb{P}}
\def\div{\textup{div}}

\newtheorem{theorem}{Theorem}[section]
\newtheorem{lemma}[theorem]{Lemma}       

\newtheorem{proposition}[theorem]{Proposition}

\begin{document}

\maketitle
\makeatletter 
\renewcommand\theequation{\thesection.\arabic{equation}}
\@addtoreset{equation}{section}
\makeatother 

\vspace{-6mm}

\begin{abstract}
In this paper we consider the It\^o SDE
  $$d X_t=d W_t+b(t,X_t)\,d t, \quad X_0=x\in\R^d,$$
where $W_t$ is a $d$-dimensional standard Wiener process and the drift coefficient
$b:[0,T]\times\R^d\to\R^d$ belongs to $L^q(0,T;L^p(\R^d))$
with $p\geq 2, q>2$ and $\frac dp +\frac 2q<1$.
In 2005, Krylov and R\"ockner \cite{KR05} proved that the above equation has a unique
strong solution $X_t$. Recently it was shown by Fedrizzi and Flandoli \cite{FF13b} that
the solution $X_t$ is indeed a stochastic flow of homeomorphisms on $\R^d$. We prove
in the present work that the Lebesgue measure is quasi-invariant under the flow $X_t$.
\end{abstract}

{\bf MSC2000:} 60H10

{\bf Key words:} Stochastic differential equation, strong solution, flow of homeomorphisms,
quasi-invariance, Zvonkin-type transformation

\section{Introduction}

Let $\sigma:[0,T]\times\R^d\to\R^m\otimes\R^d$ and $b:[0,T]\times\R^d\to\R^d$ be measurable
functions, and $W_t$ a standard Wiener process in $\R^m$. Consider the It\^o SDE
  \begin{equation}\label{sect-1.0}
  d X_t=\sigma(t,X_t)\,d W_t+b(t,X_t)\,d t, \quad X_0=x\in\R^d.
  \end{equation}
It is a classical result that if the coefficients $\sigma$ and $b$ are globally
Lipschitz continuous with respect to the spatial variable (uniformly in $t$), then
the solution $X_t$ to \eqref{sect-1.0} constitutes a stochastic flow of homeomorphisms.
In the past years, weaker conditions on the modulus of continuity of $\sigma$ and $b$,
such as log-Lipschitz continuity \cite{FL07, Zhang05a}, have been found which still
ensure the existence of a flow of homeomorphisms.

When the diffusion coefficient $\sigma$ is uniformly non-degenerate, the equation
\eqref{sect-1.0} may have pathwise uniqueness under quite weak conditions on the drift $b$.
The first result in this direction is due to Veretennikov \cite{V79}, which says that if
$\sigma(t,\cdot)$ is bounded Lipschitz continuous and satisfies a non-degeneracy condition,
then the SDE \eqref{sect-1.0} admits a unique strong solution once $b$ is bounded
measurable. In \cite{GM01}, Gy\"ongy and Martinez generalized this result to the case where
$\sigma(t,\cdot)$ is locally Lipschitz continuous, and the drift coefficient $b$ is dominated by
the sum of a positive constant and an integrable function. Their method relies on a convergence
result of the solutions of approximating SDEs to that of the limiting SDE, which follows from
the Krylov estimate. X. Zhang improved their results in \cite{Zhang05b} by replacing the
locally Lipschitz continuity of $\sigma(t,\cdot)$ with some integrability condition.

In the influential paper \cite{KR05}, Krylov and R\"ockner considered the case where
$\sigma\equiv Id$ (the identity matrix of order $d$, hence $W_t$ is now a $d$-dimensional
standard Wiener process) and the drift $b:[0,T]\times\R^d\to\R^d$ satisfies
  \begin{equation}\label{sect-1.1}
  \int_0^T\bigg(\int_{\R^d}|b(t,x)|^p\,d x\bigg)^{\frac qp}d t<+\infty
  \end{equation}
with $p\geq 2, q> 2$ such that
  \begin{equation}\label{sect-1.2}
  \frac dp +\frac 2q<1.
  \end{equation}
Hence the It\^o SDE becomes
  \begin{equation}\label{SDE}
  d X_t=d W_t+b(t,X_t)\,d t, \quad X_0=x\in\R^d.
  \end{equation}
They proved that the above equation has a unique strong solution by using Yamada--Watanabe's
criterion: existence of weak solution plus pathwise uniqueness implies the existence of a
unique strong solution. The regularity properties of functions in the Sobolev space $H_{2,p}^q(T)$
(see the next section for its definition) play an important role in the proof of pathwise
uniqueness of \eqref{SDE}. Recently, Fedrizzi and Flandoli proved that the solution $X_t$ is indeed
a stochastic flow of homeomorphisms on $\R^d$ (see \cite[Theorem 1.2]{FF13b}). Moreover, when
$v_0\in \cap_{r\geq 1} W^{1,r}(\R^d)$, they showed in \cite{FF13a} that $v(t,x):=v_0\big(X^{-1}_t(x)\big)$
is the unique weakly differentiable solution to the SPDE
  $$d v+\<b,\nabla v\>\,d t+\<\nabla v,d W_t\>=\frac12\Delta v\,d t,\quad v|_{t=0}=v_0,$$
where $\<\cdot,\cdot\>$ is the inner product in $\R^d$ (cf. \cite{MNP} for related studies).
When the dimension $d=1$ and the drift $b$ is time-independent, Aryasovay and Pilipenko
\cite{AP} obtained the Sobolev regularity of the flow $X_t$ under the assumptions that
$b$ has linear growth and locally finite variation. We mention that similar
problems were studied in \cite{FGP10a} when $b\in L^\infty(0,T; C_b^\alpha(\R^d,\R^d))$
for some $\alpha\in(0,1)$ (cf. \cite{FGP10b} for non-constant diffusion coefficients).
X. Zhang proved in \cite{Zhang11} the stochastic homeomorphism flow property for the SDE
\eqref{sect-1.0} with uniformly non-degenerate diffusion coefficient.

Our purpose in the present work is to show that the Lebesgue measure is quasi-invariant
under the flow $X_t$ generated by \eqref{SDE} with $b$ satisfying \eqref{sect-1.1}.
Here the quasi-invariance means that, almost surely, the push-forward of the Lebesgue
measure by the flow is equivalent to itself.
Recall that Fedrizzi and Flandoli \cite{FF11} proved in this case that, almost surely,
$X_t\in C^\alpha(\R^d,\R^d)$ for any $\alpha\in(0,1)$; moreover, the flow
$X_t$ is weakly differentiable in the following sense (cf. \cite[Theorem 1.2]{FF13b}):
for any $x\in\R^d$, the limit
  $$\lim_{h\to0}\frac{X_\cdot(x+he_i)-X_\cdot(x)}h$$
exists in $L^2(\Omega\times[0,T],\R^d)$, where $\{e_1,\ldots,e_d\}$ is the canonical basis
of $\R^d$. However, these regularity properties of the flow $X_t$ are not sufficient to conclude
that the Lebesgue measure is quasi-invariant under the action of $X_t$. We would like to
mention that the existence and uniqueness of generalized stochastic flow associated to
SDEs with coefficients in Sobolev spaces are studied in \cite{Zhang10, FLT10, Zhang13, Luo13},
showing that the reference measure is quasi-invariant under the flow when
the divergence and gradient of coefficients fulfill suitable (exponential) integrability.

To state the main result of this work, we denote by $\L^d$ the Lebesgue measure on $\R^d$
and $(X_t)_\#\L^d:=\L^d\circ X_t^{-1}$ the push-forward of $\L^d$ by the flow $X_t$.

\begin{theorem}\label{sect-1-thm}
Let $b:[0,T]\times\R^d\to\R^d$ be a time-dependent vector field such that \eqref{sect-1.1}
holds with $p\geq 2$ and $q>2$ satisfying \eqref{sect-1.2}. Then for all $t\in[0,T]$,
$(X_t)_\#\L^d$ is equivalent to $\L^d$ almost surely; in other words, the Lebesgue measure is
quasi-invariant under the stochastic flow $X_t$ of homeomorphisms generated by \eqref{SDE}.
\end{theorem}

We point out that it is indeed not difficult to show that, almost surely, the push-forward
$(X_t)_\#\L^d$ is absolutely continuous with respect to $\L^d$, based on the estimates
in \cite[Lemmas 3 and 5]{FF13a} (see the proof of Proposition 3.1 in the current paper).
The difficult part lies in the proof of that the Radon--Nikodym density $K_t:=\frac{d(X_t)_\#\L^d}
{d\L^d}$ is everywhere positive, that is, the two measures are equivalent. To achieve
this purpose, we shall make use of the Zvonkin-type transformation introduced in \cite{FF13a}
to get a new SDE which has more regular coefficients. We first prove that the Lebesgue
measure is quasi-invariant under the flow generated by this SDE,
then we transfer the quasi-invariance property to the original flow $X_t$.
Using this method, we do not need the existence of generalized divergence of the drift $b$,
in contrast to \cite[Theorem 1.1]{Luo11}.
The same idea does work to extend our result to the more general case
studied by X. Zhang \cite{Zhang11} (of course, we have to extend \cite[Lemmas 3 and 5]{FF13a}
to this setting). We don't want to do such technical extensions in this short note.

The organization of this paper is as follows. In Section 2 we recall some known results
which are critical for proving Theorem \ref{sect-1-thm}. In particular, we introduce the
Zvonkin-type transformation (also called It\^o--Tanaka trick) used by Fedrizzi and Flandoli
\cite{FF13a, FF13b} to prove the existence of unique strong solution to \eqref{SDE}. Following
the ideas in \cite{Luo09, Zhang10}, we first prove in Section 3 the quasi-invariance of the
flow $Y_t$ which is the strong solution to the transformed SDE \eqref{SDE-new}, then we transfer
this property to the solution $X_t$ through a $C^1$-diffeomorphism.

\section{Notations and preliminary results}

In this section we first introduce some notations of function spaces and then collect
some known results which are crucial for our paper. The main references are \cite{FF13a, FF13b}.

For $p\geq1$, $L^p(\R^d)$ denotes the usual space of (possibly vector-valued) functions
on $\R^d$ which are Lebesgue integrable of order $p$. Let
$f(t,x)$ be a function of time and space, we will use superscripts to characterize the
time-part of the norm and subscripts for the space-part: we will have $L_p^q(S,T)=
L^q(S,T;L^p(\R^d))$. For simplicity, $L_p^q(T):=L_p^q(0,T)$. We also need some notations
of Sobolev spaces: $W^{\alpha,p}(\R^d)$ is the usual Sobolev space, and
  $$\H_{\alpha,p}^q(T)=L^q(0,T;W^{\alpha,p}(\R^d)),\quad \H_p^{\beta,q}(T)=W^{\beta,q}(0,T;L^p(\R^d)).$$
Finally, $H_{\alpha,p}^q(T)=\H_{\alpha,p}^q(T)\cap \H_p^{1,q}(T)$.

Now we state the following result concerning the
existence and uniqueness of weak solutions to \eqref{SDE} (cf. \cite[Theorem 2.5 and
Corollary 2.6]{FF13b}).

\begin{theorem}\label{sect-2-thm-1}
Assume that $b\in L_p^q(T)$ with $p,q$ satisfying \eqref{sect-1.2}. Then
\begin{itemize}
\item[\rm(i)] for fixed $x\in\R^d$,
there exist processes $X_t,W_t$ defined for $t\in[0,T]$ on a filtered probability space
$(\Omega,\F,\F_t,\P)$ such that $W_t$ is a d-dimensional $(\F_t)$-Wiener process and $X_t$
is an $(\F_t)$--adapted, continuous, $d$-dimensional process for which
  \begin{equation}\label{sect-2-thm-1.1}
  \P\bigg(\int_0^T |b(t,X_t)|^2\,d t<\infty\bigg)=1
  \end{equation}
and almost surely, for all $t\in[0,T]$,
  $$X_t=x+W_t+\int_0^t b(s,X_s)\,d s.$$
\item[\rm(ii)] weak uniqueness holds for the equation \eqref{SDE} in the class of solutions
satisfying \eqref{sect-2-thm-1.1}; moreover, if $f\in L_{\tilde p}^{\tilde q}(T)$ with
$\frac d{\tilde p}+\frac 2{\tilde q}<1$, then for any $k\in\R$, there exists a constant
$C_f$ depending on $\|f\|_{L_{\tilde p}^{\tilde q}(T)}$ such that
  \begin{equation}\label{sect-2-thm-1.2}
  \sup_{x\in\R^d}\E\Big[e^{k \int_0^T|f(t,X_t)|^2\,d t}\Big]\leq C_f.
  \end{equation}
\end{itemize}
\end{theorem}

The next result (see \cite[Theorem 3.3]{FF13b}) concerning the regularity of solutions to the
backward parabolic system \eqref{sect-2-thm-2.1} plays an important role.

\begin{theorem}\label{sect-2-thm-2}
Let $\lambda>0$ and  $p\geq 2,q>2$ such that \eqref{sect-1.2} holds. Take two vector fields
$b,f\in L_p^q(T)$. Then in $H_{2,p}^q(T)$ there exists a unique solution
of the backward parabolic system
  \begin{equation}\label{sect-2-thm-2.1}
  \partial_t u+\frac12 \Delta u+b\cdot\nabla u-\lambda u +f=0,\quad u(T,x)=0.
  \end{equation}
Moreover, there exists a finite constant $N$ depending only on $d,p,q,T,\lambda$ and
$\|b\|_{L_p^q(T)}$ such that
  \begin{equation}\label{sect-2-thm-2.2}
  \|u\|_{H_{2,p}^q(T)}:=\|\partial_t u\|_{L_p^q(T)}+\|u\|_{\H_{2,p}^q(T)}\leq N\|f\|_{L_p^q(T)}.
  \end{equation}
\end{theorem}

We also have

\begin{lemma}\label{sect-2-lem-1}
Let $u_\lambda$ be the solution of \eqref{sect-2-thm-2.1}. Then
  $$\sup_{t\leq T}\|\nabla u_\lambda\|_\infty \to 0\quad \mbox{as }\lambda\to\infty,$$
where $\|\cdot\|_\infty$ is the supremum norm in the space $C(\R^d)$ of continuous functions.
\end{lemma}

In view of the above lemma, we fix $\lambda>0$ such that
  \begin{equation}\label{sect-2.1}
  \sup_{t\leq T}\|\nabla u_\lambda\|_\infty\leq \frac12.
  \end{equation}
Define
  $$\phi_\lambda(t,x)=x+u_\lambda(t,x),\quad (t,x)\in [0,T]\times\R^d.$$
The properties of the map $\phi_\lambda$ are collected in the next proposition (cf.
\cite[Lemma 3.5]{FF13b}).

\begin{proposition}\label{sect-2-prop-1}
The following statements hold:
\begin{itemize}
\item[\rm(i)] uniformly in $t\in[0,T]$, $\phi_\lambda(t,\cdot)$ has bounded first derivatives
which are H\"older continuous;
\item[\rm(ii)] for every $t\in[0,T]$, $\phi_\lambda(t,\cdot)$ is a $C^1$-diffeomorphism on $\R^d$;
\item[\rm(iii)] $\phi_\lambda^{-1}(t,\cdot):=(\phi_\lambda(t,\cdot))^{-1}$ has bounded first
spatial derivatives, uniformly in $t$;
\item[\rm(iv)] $\phi_\lambda$ and $\phi_\lambda^{-1}$ are jointly continuous in $(t,x)$.
\end{itemize}
\end{proposition}

We are now ready to state the Zvonkin-type transformation used in \cite{FGP10a, FF13a, FF13b}
to prove the existence of unique strong solution to It\^o's SDE with irregular drift coefficient.
Replacing $f$ by $b$ in the parabolic equation \eqref{sect-2-thm-2.1}, that is, we consider
  \begin{equation}\label{PDE}
  \partial_t u_\lambda+\frac12 \Delta u_\lambda+b\cdot\nabla u_\lambda
  =\lambda u_\lambda-b,\quad u_\lambda(T,x)=0.
  \end{equation}
In the following we shall fix a $\lambda>0$ such that \eqref{sect-2.1} holds, and we omit the
subscript $\lambda$ to simplify notations. Denote by $\phi_t(x)=x+u(t,x),\,(t,x)\in[0,T]\times\R^d$.
By It\^o's formula,
  \begin{align*}
  d u(t,X_t)&=\frac{\partial u}{\partial t}(t,X_t)\,d t+\nabla u(t,X_t)
  (b(t,X_t)\,d t+d W_t)+\frac12\Delta u(t,X_t)\,d t\cr
  &=\lambda u(t,X_t)\,d t-b(t,X_t)\,d t+\nabla u(t,X_t)\, d W_t.
  \end{align*}
Define the new process $Y_t=\phi_t\big(X_t(\phi_0^{-1})\big)=X_t(\phi_0^{-1})+
u\big(t,X_t(\phi_0^{-1})\big)$. Then (we write $X_t$ instead of $X_t(\phi_0^{-1})$ to save notations)
  \begin{align*}
  d Y_t&=b(t,X_t)\,d t+d W_t+\lambda u(t,X_t)\,d t-b(t,X_t)\,d t+\nabla u(t,X_t)\,d W_t\cr
  &=\lambda u(t,X_t)\,d t+(Id+\nabla u(t,X_t))\,d W_t,
  \end{align*}
where $Id$ is the identity matrix of order $d$. Since $\phi_t$ is a $C^1$-diffeomorphism,
we can define
  $$\tilde\sigma(t,y)=Id+\nabla u\big(t,\phi_t^{-1}(y)\big),\quad
  \tilde b(t,y)=\lambda u\big(t,\phi_t^{-1}(y)\big),\quad (t,y)\in [0,T]\times\R^d.$$
Therefore $Y_t$ satisfies the new It\^o's SDE
  \begin{equation}\label{SDE-new}
  d Y_t=\tilde\sigma(t,Y_t)\,d W_t+\tilde b(t,Y_t)\,d t,\quad Y_0=x.
  \end{equation}
We shall see in the next result that the coefficients of SDE \eqref{SDE-new} are more regular
than those of \eqref{SDE}, thus it is easier to be treated.

\begin{proposition}\label{sect-2-prop-2}
We have $\nabla\tilde b\in C([0,T],C_b(\R^d))$ and
  $$\tilde\sigma\in C([0,T],C_b(\R^d))\cap L^q(0,T;W^{1,p}(\R^d)).$$
\end{proposition}

Based on the regularity of the coefficients $\tilde\sigma,\tilde b$ and applying
Krylov-type estimates, Fedrizzi and Flandoli first proved the pathwise uniqueness of
solutions to the equation \eqref{SDE-new}, and then translated this result to the
solution $X_t$ of the original equation \eqref{SDE} via the transformation $\phi_t$.
Thus by Yamada--Watanabe's criterion, both equations \eqref{SDE-new} and \eqref{SDE}
have a unique strong solution. The following theorem is the main result in \cite{FF13b}
(see Theorem 1.2 there).

\begin{theorem}\label{sect-2-thm-3}
Under the assumption of Theorem \ref{sect-2-thm-1}, the equation \eqref{SDE} has a
unique strong solution $X_t$ which defines a stochastic flow of homeomorphisms and
is $\alpha$-H\"older continuous for every $\alpha<1$.
\end{theorem}

From Proposition \ref{sect-2-prop-2} and Sobolev's embedding theorem, we see that, for a.e. $t\in[0,T]$,
$\tilde\sigma(t,\cdot)$ is H\"older continuous of order $\theta<1$. Therefore we cannot
directly apply the classical results (see for instance \cite[Lemma 4.3.1]{Kunita90}) to conclude
that the solution $Y_t$ to \eqref{SDE-new} leaves the reference measure quasi-invariant.
To show our main theorem, we need to do some approximation arguments. Here are the necessary
preparations (see \cite[Lemma 12]{FF13a}).

\begin{lemma}\label{sect-2-lem-2}
Let $b^n$ be a sequence of smooth vector fields converging to $b$ in $L_p^q(T)$, and $u^n$
the solution to \eqref{PDE} with $b$ replaced by $b^n$. Then we have
\begin{itemize}
\item[\rm(i)] $u^n(t,x)$ and $\nabla u^n(t,x)$ converge pointwise in $(t,x)$ to $u(t,x)$ and
$\nabla u(t,x)$ respectively, and the convergence is uniform on compact sets;
\item[\rm(ii)] $\lim_{n\to\infty}\|u^n-u\|_{H_{2,p}^q(T)}=0$;
\item[\rm(iii)] there exists a $\lambda$ for which $\sup_{n\geq 1}\sup_{t,x}
|\nabla u^n(t,x)|\leq \frac12$;
\item[\rm(iv)] $\sup_{n\geq1}\|\nabla^2 u^n\|_{L_p^q(T)}\leq C$.
\end{itemize}
\end{lemma}

Let $Y^n_t(x)$ be the flow associated to \eqref{SDE-new} with the coefficients
  \begin{equation}\label{sect-2.2}
  \tilde \sigma^n(t,y)=Id+\nabla u^n(t,\phi_t^{n,-1}(y)),\quad
  \tilde b^n(t,y)=\lambda u^n(t,\phi_t^{n,-1}(y)),
  \end{equation}
where $\phi^n_t(x)=x+u^n(t,x)$.
For $R>0$, $B_R$ denotes the ball in $\R^d$ centered at the origin with radius $R$.
We have

\begin{proposition}\label{sect-2-prop-3}
In the situation of Lemma \ref{sect-2-lem-2}, for every $R>0$ and $k\geq 2$, we have
  $$\lim_{n\to\infty}\sup_{t\leq T}\sup_{x\in B_R}\E\big(|Y^n_t(x)-Y_t(x)|^k\big)=0.$$
The same convergence holds for the inverse flows $Y^{n,-1}_t$ and $Y^{-1}_t$.
Moreover, there exists $C_k$ independent on $n\geq 1$ such that
  $$\E\bigg[\sup_{t\leq T}|Y^n_t(x)|^k\bigg]\leq C_k(1+|x|^k).$$
\end{proposition}

\begin{proof}
The first assertion was shown in the proof of \cite[Lemma 3]{FF13a} (see p.1336),
while the second one is a consequence of \cite[Lemma 3]{FF13a} and the
relation $Y^{n,-1}_t=\phi^{n}_0\big(X^{n,-1}_t(\phi^{n,-1}_t)\big)$, where $X^n_t$
is the flow associated to \eqref{SDE} with $b$ replaced by $b_n$, and $X^{n,-1}_t$
is its inverse flow. As for the last estimate, it is a slight improvement of \cite[(26)]{FF13a}
by removing $\sup_{t\leq T}$ into the expectation: this follows from the uniform growth
of the coefficients $\tilde\sigma^n,\tilde b^n$ and classical moment estimates.
\end{proof}

\section{Proof of the main result}

In this section we first prove that the Lebesgue measure is quasi-invariant under the
stochastic flow $Y_t$ generated by the new equation \eqref{SDE-new}, following the ideas
in \cite{Luo09, Zhang10}. Then we transfer this property to the solution $X_t$ of the original
SDE \eqref{SDE} by using the diffeomorphism $\phi_t:x\mapsto x+u(t,x)$, where $u$ solves
the parabolic equation \eqref{PDE}.

We start by recalling the setting. Let $b:[0,T]\times\R^d\to\R^d$ be a time-dependent vector
field verifying the assumption of Theorem \ref{sect-1-thm}, and $u(t,x)$ the solution
to the parabolic system \eqref{PDE}. The transformation $\phi_t$ and the coefficients
$\tilde\sigma,\tilde b$ are defined as in Section 2.
As mentioned in the last section, the diffusion coefficient $\tilde\sigma$ of \eqref{SDE-new}
is only H\"older continuous, which makes it impossible to directly apply the existing results to
conclude the quasi-invariance of $Y_t:\R^d\to\R^d$. Therefore we take a sequence $\{b^n\}_{n\geq 1}$
of smooth vector fields with compact supports in $[0,T]\times\R^d$ such that
  \begin{equation}\label{sect-3.1}
  \lim_{n\to\infty}\|b^n-b\|_{L_p^q(T)}=0.
  \end{equation}
Denote by $X^n_t$ the flow of diffeomorphisms generated by \eqref{SDE} with $b$ replaced by
$b^n$, and $X^{n,-1}_t$ its inverse flow. We include the following estimate which was
proved in \cite[Lemma 5]{FF13a}: for every $k\geq 1$,
  \begin{equation}\label{Jacobian}
  \sup_{n\geq 1}\sup_{t\leq T}\sup_{x\in\R^d}\E\big(|\nabla X^{n,-1}_t(x)|^k\big)<+\infty.
  \end{equation}
Let $u^n$ be the solution to \eqref{PDE} with $b$ replaced by $b^n$. Then $u^n$ is smooth
with bounded derivatives (cf. \cite[Theorem 2]{FGP10a} where $u^n$ has bounded derivatives
up to order $2$ when $b^n\in L^\infty(0,T; C_b^\alpha(\R^d,\R^d))$; note that the parabolic
equation (9) in \cite{FGP10a} is not accompanied with the boundary condition $u(T,\cdot)=0$).
By Lemma \ref{sect-2-lem-2}(iii), we shall fix $\lambda$ big enough such that
  \begin{equation}\label{sect-3.2}
  \sup_{t\leq T}\|\nabla u(t,\cdot)\|_\infty\bigvee \sup_{n\geq 1}\sup_{t\leq T}
  \|\nabla u^n(t,\cdot)\|_\infty\leq \frac12.
  \end{equation}
Let $\tilde \sigma^n(t,y)$ and $\tilde b^n(t,y)$ be defined as in \eqref{sect-2.2}.
We consider the It\^o SDE
  $$d Y^n_t=\tilde\sigma^n(t,Y^n_t)\,d W_t+\tilde b^n(t,Y^n_t)\,d t,\quad Y^n_0=x.$$
Since the coefficients $\tilde\sigma^n(t,y)$ and $\tilde b^n(t,y)$ are smooth with bounded
spatial derivatives, uniformly in $t\in[0,T]$, we know that $Y^n_t$ is a flow of
diffeomorphisms on $\R^d$. The inverse flow is denoted by $Y^{n,-1}_t$. Moreover,
$Y^n_t=\phi^n_t\big(X^n_t(\phi^{n,-1}_0)\big)$ with $\phi^n_t(x)=x+u^n(t,x)$.

In the sequel, we denote by $(Y^{n}_t)_\#\L^d=\L^d\circ Y^{n,-1}_t$ and
$(Y^{n,-1}_t)_\#\L^d=\L^d\circ Y^n_t$ the push-forwards of the Lebesgue measure
$\L^d$ by the flows $Y^n_t$ and $Y^{n,-1}_t$. Then it is well-known that
  $$\rho^n_t:=\frac{d(Y^{n}_t)_\#\L^d}{d\L^d}=\big|\det\big(\nabla Y^{n,-1}_t\big)\big|\quad
  \mbox{and}\quad \bar\rho^n_t:=\frac{d(Y^{n,-1}_t)_\#\L^d}{d\L^d}=\big|\det\big(\nabla Y^{n}_t\big)\big|.$$
The following simple relation holds:
  \begin{equation}\label{relation}
  \rho^n_t(x)=\big[\bar\rho^n_t\big(Y^{n,-1}_t(x)\big)\big]^{-1}.
  \end{equation}
Moreover, by \cite[Lemma 4.3.1]{Kunita90} (see also \cite[(2.2)]{Zhang10}),
the density function $\bar\rho^n_t$ has the following explicit expression:
  \begin{equation}\label{density}
  \bar\rho^n_t(x)=\exp\bigg\{\int_0^t\big\<\div(\tilde\sigma^n)(s,Y^n_s(x)),d W_s\big\>
  +\int_0^t\Big[\div(\tilde b^n)-\frac12\<\nabla\tilde\sigma^n,
  (\nabla\tilde\sigma^n)^\ast\>\Big](s,Y^n_s(x))\,d s\bigg\},
  \end{equation}
where $\div(\tilde\sigma^n)=\big(\div(\tilde\sigma^n_{\cdot,1}),\ldots,\div(\tilde\sigma^n_{\cdot,d})\big)$
is a vector-valued function whose components are the divergences of the column vectors
of $\tilde\sigma^n$, and $\<\nabla\tilde\sigma^n,(\nabla\tilde\sigma^n)^\ast\>
=\sum_{k=1}^d\sum_{i,j=1}^d(\partial_i\tilde\sigma^n_{jk})(\partial_j\tilde\sigma^n_{ik})$.
Next, noting that $x=\phi^n_t(\phi^{n,-1}_t(x))=\phi^{n,-1}_t(x)+u^n(t,\phi^{n,-1}_t(x))$, thus
  $$Id=\nabla\phi^{n,-1}_t(x)+\nabla u^n(t,\phi^{n,-1}_t(x))\nabla\phi^{n,-1}_t(x).$$
As a result, for any $x\in\R^d$, we have by \eqref{sect-3.2} that ($\|\cdot\|_{op}$ is the
operator norm)
  \begin{equation*}
  1=\|Id\|_{op}\geq \big\|\nabla\phi^{n,-1}_t(x)\big\|_{op}
  -\big\|\nabla u^n(t,\phi^{n,-1}_t(x))\big\|_{op}\big\|\nabla\phi^{n,-1}_t(x)\big\|_{op}
  \geq \frac12\big\|\nabla\phi^{n,-1}_t(x)\big\|_{op},
  \end{equation*}
that is,
  \begin{equation}\label{sect-3-lem-1.3}
  \sup_{t\leq T}\big\|\nabla\phi^{n,-1}_t(x)\big\|_{op}\leq 2.
  \end{equation}
Combining this estimate with \eqref{Jacobian} and
the relation $Y^{n,-1}_t=\phi^n_0\big(X^{n,-1}_t(\phi^{n,-1}_t)\big)$, we obtain
  \begin{equation}\label{Jacobian-Y}
  \sup_{n\geq 1}\sup_{t\leq T}\sup_{x\in\R^d}\E\big(\rho^n_t(x)^k\big)=
  \sup_{n\geq 1}\sup_{t\leq T}\sup_{x\in\R^d}\E\big(\big|\det\big(\nabla Y^{n,-1}_t(x)\big)\big|^k\big)<\infty,
  \quad \mbox{for all }k\geq 1.
  \end{equation}

Now we are ready to show that the Lebesgue measure is absolutely continuous under the
action of the flow $Y_t$ generated by \eqref{SDE-new}.

\begin{proposition}[Absolute continuity under the flow $Y_t$]\label{sect-3-prop-0}
Assume the condition of Theorem \ref{sect-1-thm}. Then for any $t\in[0,T]$, the push-forward
$(Y_t)_\#\L^d$ of the Lebesgue measure is absolutely continuous with respect to $\L^d$. Moreover,
the Radon--Nikodym density $\rho_t:=\frac{d(Y_t)_\#\L^d}{d\L^d}$ satisfies
  \begin{equation}\label{sect-3-prop-0.1}
  \sup_{t\leq T}\sup_{x\in\R^d}\E\big(\rho_t(x)^k\big)<\infty, \quad \mbox{for all }k\geq 1.
  \end{equation}
\end{proposition}

\begin{proof}
Based on the estimate \eqref{Jacobian-Y} and Proposition \ref{sect-2-prop-3}, this result is
a consequence of \cite[Lemma 3.5]{Zhang10}. We include its proof here for the reader's convenience.
Proposition \ref{sect-2-prop-3} implies that, up to a subsequence, $Y^n_t(\omega,x)$ converges
to $Y_t(\omega,x)$ for $(\P\otimes\L^d)$-a.e. $(\omega,x)$ as $n\to\infty$.
We fix any $N>0$ and let $C_N(\R^d,\R_+)$ be the collection of nonnegative continuous functions with
support in $B_N$. Then for any $\varphi\in C_N(\R^d,\R_+)$, by Fubini's theorem and
Fatou's lemma, it holds for a.s. $\omega\in\Omega$ that
  \begin{equation}\label{sect-3-prop-0.2}
  \int_{\R^d}\varphi(Y_t(x))\,d x\leq \liminf_{n\to\infty}\int_{\R^d}\varphi(Y^n_t(x))\,d x
  =\liminf_{n\to\infty}\int_{\R^d}\varphi(y)\rho^n_t(y)\,d y
  =:\liminf_{n\to\infty}J^n_\varphi(\omega).
  \end{equation}
By \eqref{Jacobian-Y}, there exists a subsequence still denoted by $n$ and a $\rho^{(0)}_t\in
L^\infty(\R^d;L^k(\Omega))$ satisfying \eqref{sect-3-prop-0.1} such that
  $$\rho^n_t \mbox{ weakly}\ast \mbox{ converges to } \rho^{(0)}_t \mbox{ in }L^\infty(\R^d;L^k(\Omega)).$$
Since $\rho^n_t$ also converges weakly to $\rho^{(0)}_t$ in $L^2(\Omega\times B_N)$, by
Banach--Saks theorem, there is another subsequence still denoted by $n$ such that its
Ces\`{a}ro mean $\hat\rho^n_t:=\frac1n \sum_{k=1}^n \rho^n_t$ converges strongly to $\rho^{(0)}_t$
in $L^2(\Omega\times B_N)$. Therefore, up to a subsequence, $\hat\rho^n_t(\omega)$
converges to $\rho^{(0)}_t(\omega)$ in $L^2(B_N)$ for a.s. $\omega$. Hence
  $$\hat J^n_\varphi(\omega):=\frac1n \sum_{k=1}^nJ^n_\varphi(\omega)
  =\int_{\R^d}\varphi(y)\hat\rho^n_t(\omega,y)\,d y \longrightarrow
  \int_{\R^d}\varphi(y)\hat\rho^{(0)}_t(\omega,y)\,d y \quad \mbox{as } n\to \infty.$$
Combining this limit together with \eqref{sect-3-prop-0.2} gives us
  $$\int_{\R^d}\varphi(Y_t(x))\,d x\leq \liminf_{n\to\infty}J^n_\varphi(\omega)
  \leq \lim_{n\to\infty}\hat J^n_\varphi(\omega)=\int_{\R^d}\varphi(y)\hat\rho^{(0)}_t(\omega,y)\,d y.$$
The separability of $C_N(\R^d,\R_+)$ implies that there is a common full set $\hat\Omega$
such that the above inequality holds for all $\omega\in\hat\Omega$ and $\varphi\in C_N(\R^d,\R_+)$.
Since $N>0$ is arbitrary, we conclude that $(Y_t)_\#\L^d$ is absolutely continuous
with respect to $\L^d$ and the density function $\rho_t\leq \hat\rho^{(0)}_t$.
Hence the estimate \eqref{sect-3-prop-0.1} holds.
\end{proof}

To show that the push-forward $(Y_t)_\#\L^d$ is in fact equivalent to $\L^d$, we intend to
give in the following an explicit expression for the Radon--Nikodym density $\rho_t$.
To this end, we shall prove that the density functions $\bar\rho^n_t$
defined in \eqref{density} are convergent to
  \begin{equation}\label{density-1}
  \bar\rho_t(x)=\exp\bigg\{\int_0^t\big\<\div(\tilde\sigma)(s,Y_s(x)),d W_s\big\>
  +\int_0^t\Big[\div(\tilde b)-\frac12\<\nabla\tilde\sigma,
  (\nabla\tilde\sigma)^\ast\>\Big](s,Y_s(x))\,d s\bigg\}
  \end{equation}
in some sense. We need the technical result below.

\begin{lemma}\label{sect-3-lem-1}
Let $f\in L_{\tilde p}^{\tilde q}(T)$ with $\tilde p,\tilde q$ satisfying \eqref{sect-1.2}.
Then for any $k\geq 1$, there exists a constant $C_{k,f}$ depending on $k$ and
$\|f\|_{L_{\tilde p}^{\tilde q}(T)}$ such that
  $$\sup_{n\geq 1}\sup_{x\in\R^d}\E\Big[e^{k\int_0^T |f(t,Y^n_t(x))|^2\,d t}\Big]\leq C_{k,f}.$$
\end{lemma}

\begin{proof}
This result is a consequence of Theorem \ref{sect-2-thm-1}. Indeed, since $X^n_t$ is the
solution to \eqref{SDE} with $b$ replaced by $b^n$, then \eqref{sect-2-thm-1.2} implies
  \begin{align}\label{sect-3-lem-1.1}
  \sup_{n\geq 1}\sup_{x\in\R^d}\E\Big[e^{k\int_0^T |f(t,X^n_t(x))|^2\,d t}\Big]\leq C_{k,f},
  \end{align}
where $C_{k,f}$ depends on $k$ and $\|f\|_{L_{\tilde p}^{\tilde q}(T)}$. Recall that $Y^n_t$ is
related to $X^n_t$ by the diffeomorphism $\phi^n_t$: $Y^n_t=\phi^n_t\big(X^n_t(\phi^{n,-1}_0)\big)$.
Thus
  $$\E\Big[e^{k\int_0^T |f(t,Y^n_t(x))|^2\,d t}\Big]
  =\E\bigg[\exp\bigg(k\int_0^T \big|f\big(t,\phi^n_t\big[X^n_t(\phi^{n,-1}_0(x))\big]\big)\big|^2\,d t\bigg)\bigg].$$
Consider the new function $g^n(t,x)=f(t,\phi^n_t(x))$. By the change of variable,
  \begin{align}\label{sect-3-lem-1.2}
  \int_0^T\bigg(\int_{\R^d}|g^n(t,x)|^{\tilde p}d x\bigg)^{\frac{\tilde q}{\tilde p}}d t
  &=\int_0^T\bigg(\int_{\R^d}|f(t,y)|^{\tilde p}\big|\det\big(\nabla\phi^{n,-1}_t(y)\big)\big|
  d y\bigg)^{\frac{\tilde q}{\tilde p}}d t.
  \end{align}
Inequality \eqref{sect-3-lem-1.3} implies
  $$\sup_{n\geq 1}\sup_{y\in\R^d}\big|\det\big(\nabla\phi^{n,-1}_t(y)\big)\big|<\infty,$$
which, together with \eqref{sect-3-lem-1.2}, leads to
  $$\sup_{n\geq 1}\|g^n\|_{L_{\tilde p}^{\tilde q}(T)}<\infty.$$
Combining this estimate with \eqref{sect-3-lem-1.1} completes the proof.
\end{proof}

We now prove the following result which is analogous to \cite[Lemma 3.5]{Luo09}.

\begin{lemma}[Uniform estimate of Radon--Nikodym densities]\label{sect-3-lem-2}
For any $k\in\R$, it holds
  $$\sup_{n\geq 1}\sup_{t\leq T}\sup_{x\in\R^d}\E\big[(\bar\rho^n_t(x))^k\big]<+\infty.$$
\end{lemma}

\begin{proof}
We denote by $\xi^n=\div(\tilde b^n)
-\frac12\<\nabla\tilde\sigma^n,(\nabla\tilde\sigma^n)^\ast\>$ to simplify notations. Then
  \begin{align*}
  (\bar\rho^n_t(x))^k&=\exp\bigg\{k\int_0^t\big\<\div(\tilde\sigma^n)(s,Y^n_s(x)),d W_s\big\>
  +k\int_0^t\xi^n(s,Y^n_s(x))\,d s\bigg\}\cr
  &=\exp\bigg\{k\int_0^t\big\<\div(\tilde\sigma^n)(s,Y^n_s(x)),d W_s\big\>
  -k^2\int_0^t|\div(\tilde\sigma^n)(s,Y^n_s(x))|^2\,d s\bigg\}\cr
  &\hskip13pt \times\exp\bigg\{\int_0^t \big(k^2|\div(\tilde\sigma^n)|^2+k\xi^n\big)(s,Y^n_s(x))\,d s\bigg\}.
  \end{align*}
Thus by Cauchy's inequality,
  \begin{align}\label{sect-3-lem-2.1}
  \E\big[(\bar\rho^n_t(x))^k\big]&\leq \bigg(\E \exp\bigg\{2k\int_0^t\big\<\div(\tilde\sigma^n)(s,Y^n_s(x)),d W_s\big\>
  -2k^2\int_0^t|\div(\tilde\sigma^n)(s,Y^n_s(x))|^2\,d s\bigg\}\bigg)^{\frac12}\cr
  &\hskip13pt\times \bigg(\E\exp\bigg\{2\int_0^t \big(k^2|\div(\tilde\sigma^n)|^2+k\xi^n\big)(s,Y^n_s(x))\,d s\bigg\} \bigg)^{\frac12}.
  \end{align}
By Novikov's criterion and Lemma \ref{sect-3-lem-1}, to show that the first
exponential is a martingale, it suffices to prove that $\div(\tilde\sigma^n)\in L_p^q(T)$.
This is a consequence of Lemma \ref{sect-2-lem-2}(iv) and the definition of $\tilde\sigma^n$.
Next by \eqref{sect-3.2} and \eqref{sect-3-lem-1.3},
  \begin{align*}
  |\xi^n|\leq |\div(\tilde b^n)|+\frac12|\<\nabla\tilde\sigma^n,(\nabla\tilde\sigma^n)^\ast\>|
  &\leq \lambda \|\nabla u^n(t,\cdot)\|_\infty\|\nabla\phi^{n,-1}_t\|_\infty+\frac12 |\nabla\tilde\sigma^n|^2\cr
  &\leq C_\lambda+\frac12 |\nabla\tilde\sigma^n|^2.
  \end{align*}
Hence the second expectation in \eqref{sect-3-lem-2.1} is dominated by
  \begin{align*}
  &\hskip13pt \sup_{n\geq 1}\sup_{x\in\R^d}\E\exp\bigg\{2\int_0^T
  \big(k^2|\div(\tilde\sigma^n)|^2+|k\xi^n|\big)(s,Y^n_s(x))\,d s\bigg\}\cr
  &\leq e^{C_\lambda |k| T}\sup_{n\geq 1}\sup_{x\in\R^d}\E\exp\bigg\{(2k^2+|k|)\int_0^T \big|\nabla\tilde\sigma^n(s,Y^n_s(x))\big|^2\,d s\bigg\}<\infty,
  \end{align*}
where the last inequality follows from Lemmas \ref{sect-3-lem-1} and \ref{sect-2-lem-2}(iv).
\end{proof}

Now we show that the three integrals in the bracket on the right hand side of \eqref{density}
converge to the corresponding ones in \eqref{density-1}. First we have

\begin{proposition}\label{sect-3-prop-3}
For any $R>0$, it holds that
  $$\lim_{n\to\infty}\E\int_{B_R}\sup_{0\leq t\leq T}\bigg|
  \int_0^t\div(\tilde b^n)(s,Y^n_s(x))\,d s-\int_0^t\div(\tilde b)(s,Y_s(x))\,d s\bigg|d x=0,$$
  $$\lim_{n\to\infty}\E\int_{B_R}\sup_{0\leq t\leq T}\bigg|
  \int_0^t\<\nabla\tilde \sigma^n,(\nabla\tilde \sigma^n)^\ast\>(s,Y^n_s(x))\,d s
  -\int_0^t\<\nabla\tilde \sigma,(\nabla\tilde \sigma)^\ast\>(s,Y_s(x))\,d s\bigg|d x=0,$$
  $$\lim_{n\to\infty}\E\int_{B_R}\sup_{0\leq t\leq T}\bigg|
  \int_0^t\big\<\div(\tilde\sigma^n)(s,Y^n_s(x)),d W_s\big\>
  -\int_0^t\big\<\div(\tilde\sigma)(s,Y_s(x)),d W_s\big\>\bigg|d x=0.$$
\end{proposition}

\begin{proof}
The proofs of the three limits have some points in common, but there are minor
differences that should be taken care of. We prove them separately.

(i) We denote by $I_n$ the term on the left hand side of the first limit. Then
  \begin{align*}
  I_n&\leq \E\int_0^T\!\!\int_{B_R}\big|\div(\tilde b^n)(s,Y^n_s(x))
  -\div(\tilde b)(s,Y^n_s(x))\big|\,d xd s\cr
  &\hskip13pt +\E\int_0^T\!\!\int_{B_R}\big|\div(\tilde b)(s,Y^n_s(x))
  -\div(\tilde b)(s,Y_s(x))\big|\,d xd s
  \end{align*}
which are written as $I^n_1$ and $I^n_2$ respectively. Recall that $p\geq 2,q>2$ satisfies
\eqref{sect-1.2}. First by H\"older's inequality ($p'$ is the conjugate number of $p$),
  \begin{align*}
  I^n_1&\leq \int_0^T \bigg(\E\int_{B_R}1^{p'}\,d x\bigg)^{\frac1{p'}}
  \bigg(\E\int_{B_R}\big|\big(\div(\tilde b^n)-\div(\tilde b)\big)(s,Y^n_s(x))\big|^p
  \,d x\bigg)^{\frac1{p}}d s\cr
  &\leq C_{p,R}\int_0^T \bigg(\E\int_{\R^d}\big|\big(\div(\tilde b^n)
  -\div(\tilde b)\big)(s,y)\big|^p\rho^n_s(y)\,d y\bigg)^{\frac1{p}}d s.
  \end{align*}
Thus by \eqref{Jacobian-Y}, we have
  \begin{align*}
  I^n_1&\leq C'_{p,R}\int_0^T \bigg(\int_{\R^d}\big|\big(\div(\tilde b^n)
  -\div(\tilde b)\big)(s,y)\big|^p\,d y\bigg)^{\frac1{p}}d s\cr
  &\leq C'_{p,R,T}\bigg[\int_0^T \bigg(\int_{\R^d}\big|\big(\div(\tilde b^n)
  -\div(\tilde b)\big)(s,y)\big|^p\,d y\bigg)^{\frac q{p}}d s\bigg]^{\frac1q}\cr
  &=C'_{p,R,T} \big\|\div(\tilde b^n)-\div(\tilde b)\big\|_{L_p^q(T)}.
  \end{align*}
By the definition of $\tilde b^n,\tilde b$ and Lemma \ref{sect-2-lem-2}(ii), we conclude that
  \begin{equation}\label{sect-3-prop-3.1}
  \lim_{n\to\infty} I^n_1=0.
  \end{equation}

Now we deal the second term $I^n_2$. For any $\eps>0$, we can find $f\in C_c([0,T]\times\R^d)$
such that $\|\div(\tilde b)-f\|_{L_p^q(T)}<\eps$. Then
  \begin{align*}
  I^n_2&\leq \E\int_0^T\!\!\int_{B_R}\big|\big(\div(\tilde b)-f\big)(s,Y^n_s(x))\big|\,d xd s
  +\E\int_0^T\!\!\int_{B_R}\big|\big(f-\div(\tilde b)\big)(s,Y_s(x))\big|\,d xd s\cr
  &\hskip13pt +\E\int_0^T\!\!\int_{B_R}\big|f(s,Y^n_s(x))-f(s,Y_s(x))\big|\,d xd s
  :=I^n_{2,1}+I^n_{2,2}+I^n_{2,3}.
  \end{align*}
Analogous to the treatment of $I^n_1$, we have by H\"older's inequality and the
estimate \eqref{Jacobian-Y} that
  \begin{align*}
  I^n_{2,1}&\leq C_{p,R}\int_0^T\bigg(\E\int_{\R^d}\big|\big(\div(\tilde b)-f\big)(s,y)\big|^p
  \rho^n_s(y)\,d y\bigg)^{\frac1p}d s\cr
  &\leq C'_{p,R,T}\|\div(\tilde b)-f\|_{L_p^q(T)}<C'_{p,R,T}\eps.
  \end{align*}
In the same way, by \eqref{sect-3-prop-0.1}, we obtain $I^n_{2,2}<C'_{p,R,T}\eps$.
Next the dominated convergence theorem and Proposition \ref{sect-2-prop-3} yield
$\lim_{n\to\infty}I^n_{2,3}=0$. Summarizing these discussions, we get
$\lim_{n\to\infty}I^n_{2}=0$. Combining this result with the limit \eqref{sect-3-prop-3.1},
we obtain the first result.

(ii) We denote by $J^n$ the quantity in the second limit. Similarly we have
  \begin{align*}
  J^n&\leq \E\int_0^T\!\!\int_{B_R}\big|\big(\<\nabla\tilde\sigma^n,(\nabla\tilde\sigma^n)^\ast\>
  -\<\nabla\tilde\sigma,(\nabla\tilde\sigma)^\ast\>\big)(s,Y^n_s(x))\big|\,d xd s\cr
  &\hskip13pt +\E\int_0^T\!\!\int_{B_R}\big|\<\nabla\tilde\sigma,(\nabla\tilde\sigma)^\ast\>(s,Y^n_s(x))
  -\<\nabla\tilde\sigma,(\nabla\tilde\sigma)^\ast\>(s,Y_s(x))\big|\,d xd s
  \end{align*}
which will be denoted by $J^n_1$ and $J^n_2$. In the following we shall
assume $p>2$ (which is the case when the dimension $d\geq 2$); in fact, the case $p=2$
is simpler. Again by triangular inequality,
  \begin{align*}
  J^n_1&\leq \E\int_0^T\!\!\int_{B_R}\big|\<\nabla\tilde\sigma^n-\nabla\tilde\sigma,
  (\nabla\tilde\sigma^n)^\ast\>(s,Y^n_s(x))\big|\,d xd s\cr
  &\hskip13pt +\E\int_0^T\!\!\int_{B_R}\big|\<\nabla\tilde\sigma,
  (\nabla\tilde\sigma^n)^\ast-(\nabla\tilde\sigma)^\ast\>(s,Y^n_s(x))\big|\,d xd s
  =:J^n_{1,1}+J^n_{1,2}.
  \end{align*}
Since $p$ and $q$ are strictly greater than 2, their conjugate numbers $p',q'<2$. By H\"older's
inequality and \eqref{Jacobian-Y}, we have for all $s\in[0,T]$,
  \begin{align*}
  \E\int_{B_R}\big|\nabla\tilde\sigma^n(s,Y^n_s(x))\big|^{p'}\,d x
  &\leq C_{p,R}\bigg(\E\int_{B_R}\big|\nabla\tilde\sigma^n(s,Y^n_s(x))\big|^p\,d x\bigg)^{\frac{p'}p}\cr
  &\leq C'_{p,R}\bigg(\int_{\R^d}\big|\nabla\tilde\sigma^n(s,y)\big|^p\,d y\bigg)^{\frac{p'}p}.
  \end{align*}
Therefore
  \begin{align}\label{sect-3-prop-3.2}
  \int_0^T\bigg(\E\int_{B_R}\big|\nabla\tilde\sigma^n(s,Y^n_s(x))\big|^{p'}\,d x\bigg)^{\frac{q'}{p'}}\,d s
  &\leq C''_{p,R}\int_0^T\bigg(\int_{\R^d}\big|\nabla\tilde\sigma^n(s,y)\big|^p\,d y\bigg)^{\frac{q'}p}d s\cr
  &\leq C_{p,q,R,T}\|\nabla\tilde\sigma^n\|_{L_p^q(T)}^{q'}
  \end{align}
which, by the definition of $\tilde\sigma^n$ and Lemma \ref{sect-2-lem-2}(iv), is uniformly
bounded from above. In view of these discussions, an application of H\"older's inequality leads to
  \begin{align*}
  J^n_{1,1}&\leq \int_0^T\bigg(\E\int_{B_R}\big|\big(\nabla\tilde\sigma^n
  -\nabla\tilde\sigma\big)(s,Y^n_s(x))\big|^p\,d x\bigg)^{\frac1{p}}
  \bigg(\E\int_{B_R}\big|\nabla\tilde\sigma^n(s,Y^n_s(x))\big|^{p'}\,d x\bigg)^{\frac1{p'}}d s\cr
  &\leq C\int_0^T\bigg(\int_{\R^d}\big|\big(\nabla\tilde\sigma^n
  -\nabla\tilde\sigma\big)(s,y)\big|^p\,d y\bigg)^{\frac1{p}}
  \bigg(\E\int_{\R^d}\big|\nabla\tilde\sigma^n(s,y)\big|^{p'}\,d y\bigg)^{\frac1{p'}}d s
  \end{align*}
which, by \eqref{sect-3-prop-3.2} and Lemma \ref{sect-2-lem-2}(iv), is less than
  \begin{align*}
  C\|\nabla\tilde\sigma^n-\nabla\tilde\sigma\|_{L_p^q(T)}\|\nabla\tilde\sigma^n\|_{L_p^q(T)}
  \leq \tilde C\|\nabla\tilde\sigma^n-\nabla\tilde\sigma\|_{L_p^q(T)}.
  \end{align*}
Consequently, we deduce from Lemma \ref{sect-2-lem-2}(ii) that $\lim_{n\to\infty}J^n_{1,1}=0$.
Analogous arguments give us $\lim_{n\to\infty}J^n_{1,2}=0$. Therefore
  \begin{equation}\label{sect-3-prop-3.3}
  \lim_{n\to\infty}J^n_1=0.
  \end{equation}

Regarding the estimate of $J^n_2$, we first find a tensor-valued function
$f\in C_c([0,T]\times\R^d;(\R^d)^{\otimes3})$ such that $\|\nabla\tilde\sigma-f\|_{L_p^q(T)}<\eps$,
and then estimate it as below:
  \begin{align*}
  J^n_2&\leq \E\int_0^T\!\!\int_{B_R}\big|\big(\<\nabla\tilde\sigma,(\nabla\tilde\sigma)^\ast\>
  -\<f,f^\ast\>\big)(s,Y^n_s(x))\big|\,d xd s\cr
  &\hskip13pt +\E\int_0^T\!\!\int_{B_R}\big|\big(\<f,f^\ast\>
  -\<\nabla\tilde\sigma,(\nabla\tilde\sigma)^\ast\>\big)(s,Y_s(x))\big|\,d xd s\cr
  &\hskip13pt +\E\int_0^T\!\!\int_{B_R}\big|\<f,f^\ast\>(s,Y^n_s(x))
  -\<f,f^\ast\>(s,Y_s(x))\big|\,d xd s.
  \end{align*}
For the last term, the dominated convergence theorem yields that its limit is 0 as $n\to\infty$.
The first two terms can be dealt with as for $J^n_{1,1}$, and we conclude that they are bounded by
$C\eps$ for some $C>0$ by the choice of the function $f$. As a result, $\lim_{n\to\infty}J^n_2=0$.
Combining this with \eqref{sect-3-prop-3.3}, we see that the second limit is also zero.

(iii) Finally, we denote by $K^n$ the quantity in the last limit. By Burkholder's inequality,
  \begin{align*}
  K^n&\leq C\int_{B_R}\E\bigg[\bigg(\int_0^T \big|\div(\tilde\sigma^n)(s,Y^n_s(x))
  -\div(\tilde\sigma)(s,Y_s(x))\big|^2\,d s\bigg)^{\frac12}\bigg]d x\cr
  &\leq C_R\bigg(\int_0^T \E\int_{B_R}\big|\div(\tilde\sigma^n)(s,Y^n_s(x))
  -\div(\tilde\sigma)(s,Y_s(x))\big|^2\,d xd s\bigg)^{\frac12},
  \end{align*}
where the second inequality follows from Cauchy's inequality. It suffices to estimate
the term in the big bracket which will be denoted by $\hat K^n$. We have
  \begin{align*}
  \hat K^n&\leq 2\int_0^T \E\int_{B_R}\big|\big(\div(\tilde\sigma^n)
  -\div(\tilde\sigma)\big)(s,Y^n_s(x))\big|^2\,d xd s\cr
  &\hskip13pt +2\int_0^T \E\int_{B_R}\big|\div(\tilde\sigma)(s,Y^n_s(x))
  -\div(\tilde\sigma)(s,Y_s(x))\big|^2\,d xd s=:\hat K^n_1+\hat K^n_2.
  \end{align*}
Again we assume $p>2$ in condition \eqref{sect-1.2}. Then by H\"older's inequality
and \eqref{Jacobian-Y},
  \begin{align*}
  \hat K^n_1&\leq 2\int_0^T \bigg(\E\int_{B_R}1\,d x\bigg)^{1-\frac2p}
  \bigg(\E\int_{B_R}\big|\big(\div(\tilde\sigma^n)
  -\div(\tilde\sigma)\big)(s,Y^n_s(x))\big|^p\,d x\bigg)^{\frac2p}d s\cr
  &\leq C_{p,R}\bigg(\E\int_{\R^d}\big|\big(\div(\tilde\sigma^n)
  -\div(\tilde\sigma)\big)(s,y)\big|^p\,d y\bigg)^{\frac2p}d s\cr
  &\leq C_{p,R,q,T}\|\div(\tilde\sigma^n)-\div(\tilde\sigma)\|_{L_p^q(T)}^2
  \end{align*}
which, due to Lemma \ref{sect-2-lem-2}(ii), goes to 0 as $n$ increases to infinity.
The treatment of $\hat K^n_2$ is analogous to that of $I^n_2$, hence we omit it. The proofs
are finally completed.
\end{proof}

We are at the position of proving

\begin{proposition}[Quasi-invariance under the flow $Y_t$]\label{sect-3-prop-4}
For any $t\in[0,T]$, the push-forward $(Y_t)_\#\L^d$ of the Lebesgue measure $\L^d$ by the
flow $Y_t$ is equivalent to $\L^d$; moreover,
  $$\rho_t(x):=\frac{d(Y_t)_\#\L^d}{d\L^d}(x)=\big[\bar\rho_t\big(Y^{-1}_t(x)\big)\big]^{-1},$$
where the Radon--Nikodym density $\bar\rho_t(x)$ is defined in \eqref{density-1}.
\end{proposition}

\begin{proof}
By Proposition \ref{sect-3-prop-3}, there is a subsequence still denoted by $n$ such that for
$(\P\times\L^d)$-almost all $(\omega,x)$,
  \begin{equation*}
  \lim_{n\to\infty}\bar\rho^n_t(\omega,x)=\bar\rho_t(\omega,x) \quad\mbox{uniformly in } t\in[0,T].
  \end{equation*}
Taking into account the uniform bound proved in Proposition \ref{sect-3-lem-2},
we have for any $k\geq 1$ and $R>0$ that
  \begin{equation}\label{sect-3-prop-4.1}
  \lim_{n\to\infty}\E\int_{B_R}\big|\bar\rho^n_t(x)-\bar\rho_t(x)\big|^k\,d x=0.
  \end{equation}
Now for any $\varphi,\psi\in C_c(\R^d,\R_+)$, we have $\P$-a.s.,
  $$\int_{\R^d}\varphi\big(Y^{n,-1}_t(x)\big)\psi(x)\,d x
  =\int_{\R^d}\varphi(y)\psi\big(Y^{n}_t(y)\big)\bar\rho^n_t(y)\,d y\quad \mbox{for all }t\leq T.$$
By \eqref{sect-3-prop-4.1} and Proposition \ref{sect-2-prop-3}, up to a subsequence, the
two sides of the above equality are convergent in $L^1(\P)$ for any fixed $t\in[0,T]$. Thus
we get for $\P$-a.s. $\omega$,
  \begin{equation}\label{sect-3-prop-4.2}
  \int_{\R^d}\varphi\big(Y^{-1}_t(x)\big)\psi(x)\,d x
  =\int_{\R^d}\varphi(y)\psi\big(Y_t(y)\big)\bar\rho_t(y)\,d y.
  \end{equation}
Since $C_c(\R^d,\R_+)$ is separable, we can find a full set $\Omega_t\subset\Omega$
such that for all $\omega\in\Omega_t$, the above identity holds for any $\varphi,\psi\in
C_c(\R^d,\R_+)$. Noting that $\P$-a.s., $\bar\rho_t(y)$ is positive for all $(t,y)\in
[0,T]\times\R^d$, we finish the proof by applying \cite[Lemma 3.4(ii)]{Zhang10}.
\end{proof}

Finally we can prove the main result of this paper.

\begin{proof}[Proof of Theorem \ref{sect-1-thm}]
Fix any $\varphi\in C_c(\R^d)$. Since $X_t=\phi^{-1}_t(Y_t(\phi_0))$, we have $\P$-a.s.,
  \begin{align*}
  \int_{\R^d}\varphi(X_t(x))\,d x&=\int_{\R^d}\varphi\big[\phi^{-1}_t\big(Y_t(\phi_0(x))\big)\big]\,d x
  =\int_{\R^d}\varphi\big[\phi^{-1}_t\big(Y_t(y)\big)\big]\cdot\big|\det(\nabla\phi^{-1}_0(y))\big|\,d y.
  \end{align*}
Applying Propositions \ref{sect-3-prop-4} and \ref{sect-3-prop-0} leads to
  \begin{align*}
  \int_{\R^d}\varphi(X_t(x))\,d x&=\int_{\R^d}\varphi\big(\phi^{-1}_t(x)\big)\rho_t(x)
  \big|\det\big[\nabla\phi^{-1}_0\big(Y^{-1}_t(x)\big)\big]\big|\,d x\cr
  &=\int_{\R^d}\varphi(y)\rho_t(\phi_t(y))\big|\det(\nabla\phi_t(y))\big|\cdot
  \big|\det\big[\nabla\phi^{-1}_0\big(Y^{-1}_t(\phi_t(y))\big)\big]\big|\,d y\cr
  &=\int_{\R^d}\varphi(y)\big[\bar\rho_t\big(Y^{-1}_t(\phi_t(y))\big)\big]^{-1}
  \big|\det(\nabla\phi_t(y))\big|\cdot\big|\det\big[\nabla\phi^{-1}_0\big(Y^{-1}_t(\phi_t(y))\big)\big]\big|\,d y.
  \end{align*}
Therefore, for $\P$-a.s. $\omega$, $(X_t)_\#\L^d$ is absolutely continuous with respect to
$\L^d$ with the Radon--Nikodym density
  $$K_t(x):=\frac{d(Y_t)_\#\L^d}{d\L^d}(x)=\big[\bar\rho_t\big(Y^{-1}_t(\phi_t(x))\big)\big]^{-1}
  \big|\det(\nabla\phi_t(x))\big|\cdot\big|\det\big[\nabla\phi^{-1}_0\big(Y^{-1}_t(\phi_t(x))\big)\big]\big|.$$
Noticing that $\bar\rho_t$ is positive everywhere and $\nabla\phi_t(x)$ is non-degenerate
for all $x\in\R^d$, we see that the Radon--Nikodym density $K_t(x)$ is positive for all
$x\in\R^d$. Consequently, $(X_t)_\#\L^d$ is equivalent to $\L^d$; in other words,
the Lebesgue measure is quasi-invariant under the action of the flow $X_t$ generated by \eqref{SDE}.
\end{proof}

\end{document}